\documentclass[12pt,epsfig]{article}

\usepackage{amsmath, amssymb, amsthm}
\usepackage[margin=1.2 in]{geometry}

\usepackage[pdftex]{graphicx}
\numberwithin{equation}{section}
\newtheorem{thm}{Theorem}[section]
\newtheorem{lem}{Lemma}[section]
\newtheorem{rem}{Remark}[section]
\newtheorem*{rem*}{Remark}

\title %[Cahn-Hilliard equations]
{On the energy stability of  Strang-splitting for Cahn-Hilliard}
\date{}

\author
{Dong Li
\thanks
{Department of Mathematics, the Hong Kong University of Science
\& Technology, Clear Water Bay, Hong Kong.
 Email: {mpdongli@gmail.com}.
 }\qquad
{Chaoyu Quan}	
\thanks{SUSTech International Center for Mathematics, Southern University of Science and Technology,	Shenzhen, China.
Email: {quancy@sustech.edu.cn}.
}
}

\begin{document}
\maketitle
\begin{abstract}
We consider a Strang-type second
order operator-splitting discretization for the Cahn-Hilliard equation.  We introduce a new theoretical framework
and prove uniform energy stability of the numerical solution and  persistence of all higher
Sobolev norms. This is the first strong stability result for second order operator-splitting 
methods for the Cahn-Hilliard equation. In particular we  settle several long-standing open issues in the work of Cheng, Kurganov, Qu and Tang \cite{Tang15}.
\end{abstract}
%\maketitle
\section{Introduction}
In this work we consider  the  Cahn-Hilliard equation (\cite{CH58}) of the form:
\begin{align} \label{1}
\begin{cases}
\partial_t u  = \Delta ( -\nu \Delta u +f (u) ), \quad (t,x) \in (0, \infty) \times \Omega, \\
u \Bigr|_{t=0} =u_0.
\end{cases}
\end{align}
Here  the main unknown $u=u(t,x):\, [0,\infty)\times \Omega \to \mathbb R$ denotes the concentration difference in a binary system. The parameter $\nu>0$ is called the mobility coefficient and we fix it  as a constant for simplicity.  We take the nonlinear term
$f(u)=u^3-u =F^{\prime}(u)$, where $F(u) = \frac 14 (u^2-1)^2$ is the standard double well.
The  minima of this potential are situated  at $u=\pm 1$ which correspond to different phases or states. In order not to overburden the readers
with various subtle technicalities,  we  take the spatial domain $\Omega$ in \eqref{1} as the  one-dimensional $2\pi$-periodic torus $\mathbb T=\mathbb R/ 2\pi \mathbb Z
=[-\pi,\pi]$.  With some additional work our analysis can be extended to other physical dimensions $d\le 3$.  Throughout this note we shall consider mean zero initial data, i.e. $
\frac 1{2\pi} \int_{\mathbb T} u_0 dx =0$. This is clearly
invariant under the dynamics thanks to the mass conservation law.
It follows that $u(t,\cdot)$ has zero mean for all $t>0$.  
The system \eqref{1} naturally arises as a gradient flow of a Ginzburg-Landau type energy
functional $ E(u)$ in $H^{-1}$, where
\begin{equation}
 E(u)= \int_{\Omega} \left( \frac 12 \nu |\nabla u|^2 + F(u) \right) dx
=\int_{\Omega} \left( \frac 12 \nu |\nabla u|^2 + \frac 14 (u^2-1)^2 \right) dx.
\end{equation}
For smooth solutions, the fundamental energy conservation law can be expressed as
\begin{equation}
\frac d {dt}  E ( u(t) ) +
\| |\nabla|^{-1}  \partial_t u \|_2^2
=\frac d {dt}  E(u(t)) + \int_{\Omega}
| \nabla ( -\nu \Delta u + f(u ) ) |^2 dx =0.
\end{equation}
Consequently, one obtains a priori $\dot H^1$-norm control of the solution for all $t>0$. 
Since the scaling-critical space for CH is $H^{-\frac 12}$ in 1D, the global wellposedness and regularity  for $H^1$-initial data follows from standard arguments. 

%For $\tau>0$, we let $S_L(\tau)= e^{-\tau \nu \Delta^2}$ be the exact solution operator to the linear equation:
%\begin{align}
%\partial_t  u= - \nu \Delta^2 u.
%\end{align}
%We define $S_N(\tau): w \to u$ as the solution operator to the following problem:
%\begin{align}
%\frac{u-w} {\tau } = \Delta ( w^3 - w).
%\end{align}
%In yet other words,
%\begin{align}
%u =S_N(\tau) w = w+ \tau \Delta (w^3-w).
%\end{align}
%This is one of the simplest discretization on the timer interval $[0, \tau]$  for the exact problem 
%\begin{align} \label{1.11}
%\begin{cases}
%\partial_t  u = \Delta ( u^3 -u), \qquad \quad t>0; \\
%u \bigr|_{t=0} =w.
%\end{cases}
%\end{align}
%By using the operator-splitting, the solution of the original equation from time $t$ to
%time $t+\tau$ is approximated as
%\begin{align}
%u(t+\tau, x) 
%\approx 
%\Bigl(  S_L(  \tau ) S_N( \tau)  u \Bigr) (t,x).
%\end{align}

The main purpose of this work is establish strong stability of a second order strang-type operator-splitting
algorithm applied to the Cahn-Hilliard equation \eqref{1}.  Concerning the operator splitting
approximation of \eqref{1}, there is a lot of flexibility in designing the linear/nonlinear operators
and the interwoven patterns of these operators.  To fix the terminology, let $a$ be a fixed initial
data, and define for $\tau>0$,  $S_L( \tau) a = e^{-\tau (\nu \Delta^2+\Delta)} a$.
In yet other words, $u= S_L(t) a$ solves the equation
\begin{align} \label{G1.4}
\begin{cases}
\partial_t u = -\nu \Delta^2 u - \Delta u,  \quad 0<t\le \tau;\\
u \Bigr|_{t=0} = a.
\end{cases}
\end{align}
Let $u=S_N(t) a$ solve the nonlinear problem
\begin{align} \label{G1.5} 
\begin{cases}
\partial_t u = \Delta (u^3), \quad 0<t\le \tau;\\
u \Bigr|_{t=0} =a.
\end{cases}
\end{align}
Denote by $u=u^{\mathrm{P}}$ the exact PDE solution to \eqref{1} corresponding
to initial data $a$. A Strang-type
approximation amounts to the approximation of the form:
\begin{align} \label{G1.6}
u^{\mathrm{P}} (\tau,\cdot) =  S_L(\frac {\tau}2) S_N(\tau) S_N(\frac{\tau}2) a  +O(\tau^3).
\end{align}
At the cost of some high regularity assumption on $a$ and certain smallness of the time interval length
$\tau$, one can show that \eqref{G1.6} holds in some Sobolev class.  In practical numerical
implementation, we need to iterate \eqref{G1.6} $n$-times, that is
\begin{align}
u^{\mathrm{P}}(n\tau, \cdot) \approx 
\underbrace{S_L(\frac {\tau}2) S_N(\tau) S_N(\frac{\tau}2)\cdots S_L(\frac {\tau}2) S_N(\tau) S_N(\frac{\tau}2)} _{\text{$n$ times}} a.
\end{align}
To justify the convergence and stability of the numerical approximation, a fundamental
problem is to establish an estimate of the form 
\begin{align}
\sup_{n\ge 1} \| \underbrace{S_L(\frac {\tau}2) S_N(\tau) S_N(\frac{\tau}2)\cdots S_L(\frac {\tau}2) S_N(\tau) S_N(\frac{\tau}2)} _{\text{$n$ times}} a \|_{H^k} \lesssim  1,
\end{align}
where we assume $a \in H^k$ for some $k\ge 1$, and $\tau$ is taken to be sufficiently small. 
This is by no means trivial since \eqref{G1.4} in general only guarantees 
$\| S_L(\frac {\tau} 2 ) a \|_2 \le e^{c \tau} \| a\|_2$ and the nonlinear evolution \eqref{G1.5} only gives control of the $L^p$-norm. Needless to say, the wellposedness of \eqref{G1.5} in Sobolev class and control of the lifespan of the local solution also present various subtle technical difficulties. 
Another variation of the theme for the operator-splitting approximation of \eqref{1} goes as follows.
 Define for $\tau>0$,  $S^{(1)}_L( \tau) a = e^{-\nu \tau \Delta^2} a$, i.e.
  $u= S^{(1)}_L(t) a$ solves the equation
\begin{align} 
\begin{cases}
\partial_t u = -\nu \Delta^2 u,  \quad 0<t\le \tau;\\
u \Bigr|_{t=0} = a.
\end{cases}
\end{align}
Let $u=S^{(1)}_N(t) a$ solve the nonlinear problem
\begin{align} 
\begin{cases}
\partial_t u = \Delta (u^3-u), \quad 0<t\le \tau;\\
u \Bigr|_{t=0} =a.
\end{cases}
\end{align}
We then approximate $u^{\mathrm{P}}(\tau, \cdot)$ via the scheme
\begin{align}
u^{\mathrm{P}}(\tau, \cdot) \approx
S^{(1)}_L(\frac {\tau}2) S^{(1)}_N(\tau) S^{(1)}_L(\frac {\tau}2) a.
\end{align}
One should note that although we have $\| S^{(1)}_L(\tilde  \tau) a\|_{H^k} \le 
\|a \|_{H^k}$ for any $k\ge 0$. The nonlinear evolution $S^{(1)}_N(\tau)$ no longer
has contraction in $L^p$. This brings essential technical difficulties for the stability analysis.

Due to these aforementioned technical obstructions, there were very few rigorous results on the analysis of the operator-splitting
type algorithms for the Cahn-Hilliard equation and similar models\footnote{Most results in the
literature are conditional one way or another in disguise.}. In \cite{Red19}, Gidey and Reddy considered a convective Cahn-Hilliard
model of the form
\begin{align} \label{1.14}
\partial_t u - \gamma \nabla \cdot \mathbf{h}(u) + \epsilon^2 \Delta^2 u
=\Delta (f(u)),
\end{align}
where $\mathbf{h}(u) =\frac 12 (u^2, u^2)$. 
By using operator-splitting , \eqref{1.14} were split  into the hyperbolic part, nonlinear diffusion part and diffusion part respectively. Some conditional results  concerning certain weak solutions were obtained in \cite{Red19}.
  In \cite{Feng19},
Weng, Zhai and Feng considered a viscous Cahn-Hilliard model of the form
\begin{align}
(1-\alpha) \partial_t u = \Delta ( - \epsilon^2 \Delta u + f(u ) + \alpha \partial_t u),
\end{align}
where the parameter $\alpha \in (0, 1)$. Weng, Zhai and Feng considered a fast explicit Strang splitting and showed
stability and convergence under the assumption that $A=\|\nabla u^{\operatorname{num}}\|^2_{\infty}$, $B=
\| u^{\operatorname{num}} \|_{\infty}^2$ are bounded,  and satisfy a technical condition
$6A+8-24B>0$ (see Theorem 1 on pp. 7 of \cite{Feng19}), where $u^{\mathrm{num}}$ denotes
the numerical solution.

The first genuine progress on the energy-stability analysis of the operator-splitting approximation 
of \eqref{1} were made  in recent \cite{21osCH},  where we considered a splitting approximation 
of \eqref{1} of the form:
\begin{align}
u^{\mathrm{P}}(\tau, \cdot) \approx S_L^{(1)}(\tau) S_N^{(2)}(\tau) a.
\end{align}
Here $u= S_N^{(2)}(\tau) a$ solves
\begin{align}
\frac {u - a} {\tau} =   \Delta ( a^3 -a).
\end{align}
By introducing a novel modified energy, we showed monotonic decay of the new modified
energy which is coercive in $H^1$-sense.  Moreover we also obtained uniform control of higher Sobolev regularity. However, this line of analysis relies in an essential way the monotonicity of the
modified energy and has no bearing on the
second-order and higher case which have some intrinsic technical difficulties.  In \cite{Tang15}, Cheng, Kurganov, Qu and Tang considered the Strang splitting
for the Cahn-Hilliard equation in the style of \eqref{G1.6}.  Some conditional results were given
in \cite{Tang15} but the rigorous analysis of energy stability has remained an outstanding open problem.
The purpose of this work is to establish a completely new theoretical framework for the rigorous
analysis of energy stability and higher-order Sobolev-norm stability for higher order operator-splitting
method such as \eqref{G1.6}. Our first result reads as follows.
 
\begin{thm}\label{thm0}
Let $\nu>0$ and consider the one-dimensional periodic torus $\mathbb T
=[-\pi, \pi]$.  Assume the initial data $u^0
\in H^{k_0}(\mathbb T)$ ($k_0\ge 1$ is an integer) and has mean zero.  Let $\tau>0$ and define
\begin{align}
u^{n+1} = S_L( \frac {\tau}2 ) S_N(\tau )   S_L(\frac {\tau}2) u^n, \quad n\ge 0.
\end{align}
There exists a constant $\tau_*>0$ depending only on $\|u^0\|_{2}$ and
$\nu$, such that if $0<\tau <\tau_*$, then
\begin{align}
\sup_{n\ge 0} \| u^n \|_{H^{k_0}} \le A_1<\infty,
\end{align}
where $A_1>0$ depends on ($\| u^0\|_{H^{k_0} }$, $\nu$, $k_0$). 
\end{thm}

Our second result establishes the  convergence of the operator splitting approximation.
Not surprisingly since this is a Strang-type splitting approximation, the convergence is second order in $\tau$ on any finite time interval $[0,T]$. 
\begin{thm}[Convergence of the splitting approximation] \label{thm1}
Assume the initial data $u^0 \in H^{40}(\mathbb T)$ with mean zero. 
Let $u^n$ be defined as in Theorem \ref{thm0}.  Let $u$ be the exact PDE solution
to \eqref{1}  corresponding to initial data $u^0$. Let $0<\tau <\tau_*$ as in Theorem
\ref{thm0}. Then for any $T>0$, we have
\begin{align}
\sup_{n\ge 1, n\tau \le T}  \| u^n - u(n\tau, \cdot ) \|_{L^2(\mathbb T)}
\le C \cdot \tau^2,
\end{align}
where $C>0$ depends on ($\nu$, $\|u^0\|_{H^{40}}$, $T$).
\end{thm}
\begin{rem}
The regularity assumption on initial data can be lowered but we shall not dwell on this
issue here for simplicity of presentation. One can also work out the convergence in
higher Sobolev norms. 
\end{rem}
The rest of this paper is organized as follows. In Section $2$ we set up the notation and collect
some preliminary lemmas.  In Section $3$ we carry out the main analysis for the propagators.
In Section $4$ we complete the proofs of Theorem \ref{thm0} and \ref{thm1}.

\section{Notation and preliminaries}

For any two positive quantities $X$ and $Y$, we shall write $X\lesssim Y$ or $Y\gtrsim X$ if
$X \le  CY$ for some  constant $C>0$ whose precise value is unimportant.
We shall write $X\sim Y$ if both $X\lesssim Y$ and $Y\lesssim X$ hold.
We write $X\lesssim_{\alpha}Y$ if the
constant $C$ depends on some parameter $\alpha$.
We shall
write $X=O(Y)$ if $|X| \lesssim Y$ and $X=O_{\alpha}(Y)$ if $|X| \lesssim_{\alpha} Y$.

We shall denote $X\ll Y$ if
$X \le c Y$ for some sufficiently small constant $c$. The smallness of the constant $c$ is
usually clear from the context. The notation $X\gg Y$ is similarly defined. Note that
our use of $\ll$ and $\gg$ here is \emph{different} from the usual Vinogradov notation
in number theory or asymptotic analysis.

For any $x=(x_1,\cdots, x_d) \in \mathbb R^d$, we denote $|x| =|x|_2=\sqrt{x_1^2+\cdots+x_d^2}$, and
$|x|_{\infty} =\max_{1\le j \le d} |x_j|$.
Also occasionally we use the Japanese bracket notation:
$\langle x \rangle =(1+|x|^2)^{\frac 12}.$

We denote by $\mathbb T^d=[-\pi, \pi]^d = \mathbb R^d/2\pi \mathbb Z^d$ the usual
$2\pi$-periodic torus.
For $1\le p \le \infty$ and any function $f:\, x\in \mathbb T^d \to \mathbb R$, we denote
the Lebesgue $L^p$-norm of $f$ as
\begin{align*}
\|f \|_{L^p_x(\mathbb T^d)} =\|f\|_{L^p(\mathbb T^d)} =\| f \|_p.
\end{align*}
If $(a_j)_{j \in I}$ is a sequence of complex numbers
and $I$ is the index set, we denote the discrete $l^p$-norm
as
\begin{equation}
\| (a_j) \|_{l_j^p(j\in I)} = \| (a_j) \|_{l^p(I)} =
\begin{cases}
 {\displaystyle \left(\sum_{j\in I} |a_j|^p\right)^{\frac 1p}}, \quad 0<p<\infty, \\
 \sup_{j\in I} |a_j|, \quad \qquad p=\infty.
 \end{cases}
 \end{equation}
 For example,
$ \| \hat f(k) \|_{l_k^2(\mathbb Z^d)} = \left(\sum_{k \in \mathbb Z^d} |\hat f(k)|^2\right)^{\frac 12}$.
If $f=(f_1,\cdots,f_m)$ is a vector-valued function, we denote
$|f| =\sqrt{\sum_{j=1}^m |f_j|^2}$, and
$\| f\|_p = \| ({\sum_{j=1}^m f_j^2})^{\frac 12} \|_p$.
We use similar convention for the corresponding discrete $l^p$ norms for the vector-valued
case.

%We denote
%\begin{align*}
%\operatorname{sgn}(x) =\begin{cases}
%1, \quad x>0;\\
%-1, \quad x<0.
%\end{cases}
%\end{align*}

We use the following convention for the Fourier transform pair:
\begin{equation} \label{eqFt2}
\hat f(k) = \int_{\mathbb T^d} f(x) e^{- i k\cdot x} dx, \quad
 f(x) =\frac 1 {(2\pi)^d}\sum_{k\in \mathbb Z^d} \hat f(k) e^{ ik \cdot x},
\end{equation}
and denote for $0\le s \in \mathbb R$,
\begin{subequations}
\begin{align}
&\|f \|_{\dot H^s} = \|f \|_{\dot H^s(\mathbb T^d)} = \| |\nabla|^s f \|_{L^2(\mathbb T^d)}
\sim \|  |k|^s \hat f (k) \|_{l^2_k (\mathbb Z^d)}, \\
& \| f \|_{H^s} = \sqrt{\| f \|_2^2 + \| f\|_{\dot H^s}^2}  \sim \| \langle
 |k| \rangle^s \hat f(k) \|_{l^2_k(\mathbb Z^d)}.
\end{align}
\end{subequations}

\begin{lem} \label{leKbeta}
Let $\nu>0$, $d\ge 1$ and $\beta>0$. Consider on the torus $\mathbb T^d=[-\pi, \pi]^d$,
\begin{equation}
K(x) = \mathcal F^{-1} ( e^{-\beta (\nu |k|^4 -|k|^2)} )
=e^{-\beta(\nu \Delta^2+\Delta) } \delta_0,
\end{equation}
 where $\delta_0$ is the periodic Dirac comb. Then for any $1\le p\le \infty$,
\begin{align}
\| K \|_{L^p(\mathbb T^d)} \le c_{d,p, \nu}\,  (1+\beta^{-d(\frac 14 -\frac 1{4p})})
e^{d_1 \beta},
\end{align}
where $c_{d,p, \nu}>0$ depends only on ($d$, $p$, $\nu$) and $d_1>0$ depends only
on ($d$, $\nu$).  
\end{lem}
\begin{proof}[Proof of Lemma \ref{leKbeta}]
We shall write $X\lesssim Y$ if $X \le C Y$ and $C$ depends on ($d$, $\nu$, $p$).

Define $K_w(x) = (2\pi)^{-d}  \int_{\mathbb R^d}
e^{ i \xi \cdot x} e^{-\beta(\nu |\xi|^4 -|\xi|^2)} d\xi$.  It is not difficult to check that
\begin{align}
|K_w(x) | \lesssim \langle x \rangle^{-10d} e^{d_1\beta} (1+\beta^{-\frac d4} ).
\end{align}

Poisson summation gives
\begin{align} \label{2.8Ta}
K(x)  =\sum_{l \in \mathbb Z^d} K_w( x+ 2\pi l).
\end{align}
The desired estimate then follows easily.
\end{proof}

\begin{lem} \label{lem2.2}
Let  $d=1$ and $\nu>0$. Let $0<\tau\le 1$. Then for any $g\in L^4(\mathbb T)$,
we have
\begin{align}
&\| e^{- \tau  (\nu \partial_x^4 +\partial_x^2)}  g \|_{\infty} \le C_1 \tau^{-\frac 1{16} }
\|g\|_4,
\end{align}
where $C_1>0$ depends on $\nu$. 
For any $g_1\in L^{\frac 43}(\mathbb T)$, we have
\begin{align}
& \| \tau \partial_{xx} e^{- \tau (\nu  \partial_x^4 +\partial_x^2)} g_1\|_{\infty}
\le C_2  \tau^{\frac 5 {16}} \| g_1\|_{\frac 43},
\end{align}
where $C_2>0$ depends on $\nu$.
\end{lem}

\begin{proof}
The first inequality follows from Lemma
\ref{leKbeta}. For the second inequality denote
\begin{align}
K_2=\mathcal F^{-1}\left( \tau |k|^2 e^{-\tau (\nu |k|^4-k^2)}  \right).
\end{align}
We then have $\| K_{2} \|_{L_x^4(\mathbb T)}\lesssim \| \widehat{K_{2} } \|_{l_k^{\frac 43} (\mathbb Z)} \lesssim
\; \tau^{-\frac 5{16}}$ for $0<\tau\le 1$.
\end{proof}

\section{Analysis of the propagators}

\begin{lem}[One-step solvability of $S_N(\tau)$] \label{lemG1}
Let $\nu>0$. Suppose $a\in L^2(\mathbb T)$ with zero mean and $\| a \|_2 \le A_1<\infty$ for some $A_1>0$. There exists
$\tau_1=\tau_1(A_1, \nu)>0$ such that if $0<\tau \le \tau_1$, then there exists a unique solution
$w\in  C([0,\tau], H^2)$ to the equation
\begin{align}
\begin{cases}
\partial_t w = \partial_{xx} (w^3), \quad \, 0<t\le \tau, \\
w\Bigr|_{t=0} =e^{-\frac 12 \tau  (\nu \partial_x^4-\partial_x^2)} a.
\end{cases}
\end{align}
The solution $w$ satisfies 
\begin{align}
&\max_{0\le t \le \tau} \| w(t, \cdot ) \|_{2} \le  \| w(0,\cdot )\|_2 
\le  e^{c_{\nu} \tau} \|a\|_2, \notag \\
& \max_{0\le t\le \tau} \| w(t,\cdot) \|_4 \le \| w(0,\cdot) \|_4 \le
c_{\nu}^{(1)} (1+\tau^{-\frac 1{16} }) \| a\|_2, \notag 
\end{align}
where $c_{\nu}\ge 0$, $c_{\nu}^{(1)}>0$ depend only on $\nu$. Furthermore if $\| a\|_{H^k(\mathbb T)} \le A_2<\infty$
for some integer $k\ge 1$ and $A_2>0$, then we also have the bound
\begin{align} \label{G5.3}
\sup_{0\le t\le \tau} \| w(t,\cdot) \|_{H^k} \le C^{(1)}_{\nu, k, A_1, A_2} <\infty,
\end{align}
where $C^{(1)}_{\nu, k, A_1, A_2}>0$ depends on ($\nu$, $k$, $A_1$, $A_2$). 
\end{lem}
\begin{proof}
For the local existence of solution, we can work with a regularized problem
$\partial_t w_{\delta} = -\delta \partial_x^4 w_{\delta} + \partial_{xx} ( (w_{\delta})^3)$
and take the limit $\delta \to 0$.  The key point in the argument is to derive uniform-in-$\delta$ bounds
on $\| \partial_{xx} w \|_2$.  Here and below we drop the subscript $\delta$ for simplicity.  Thanks
to the benign nonlinear diffusion term, it is not difficult to check that 
\begin{align}
\frac d {dt} \| \partial_{xx} w \|_2^2 \lesssim \| w \|_{H^2}^3.
\end{align}
Note that 
\begin{align}
\| w(0,\cdot)\|_{H^2} \lesssim (1+\tau^{-\frac 12} ) \|a\|_2.
\end{align}
It suffices for us to choose $\tau$ sufficiently small such that
\begin{align}
\tau (1+\tau^{-\frac 12} ) \| a\|_2 \ll 1.
\end{align}
We then obtain a local solution in $C([0,\tau], H^2)$. It is not difficult to check the uniqueness and 
the $L^2$, $L^4$ estimates.
The estimate \eqref{G5.3} follows from the $H^2$ bound and additional energy estimates. We omit
the details.
\end{proof}

\begin{lem}[$O(\frac 1 {\tau})$-step stability of 
$S_N(\tau)S_L(\frac{\tau}2)$ and 
$S_L(\frac {\tau}2)S_N(\tau) S_L(\frac {\tau}2)$] \label{lemG2}
Let $\nu>0$. Suppose $a\in L^2(\mathbb T)$ with zero mean and $\| a \|_2 \le A_1<\infty$ for some $A_1>0$.  Define $u^0=v^0 =a$ and
\begin{align}
& v^{n+1}=S_N(\tau) S_L(\frac {\tau}2) v^n, \quad n\ge 0; \notag \\
&u^{n+1} = S_L(\frac {\tau}2) S_N(\tau) S_L(\frac{\tau}2) u^n, \quad n\ge 0. \notag 
\end{align}
There exists
$\tau_2=\tau_2(A_1, \nu)>0$ such that if $0<\tau \le \tau_2$, then the following hold.
\begin{enumerate}
\item  The iterates $u^n$, $v^n$ are well-defined for all $n\le \frac {10}{\tau}$, and 
\begin{align} \label{G5.8}
\sup_{1\le n \le \frac {10}{\tau}} ( \| u^n \|_2 + \| v^n \|_2) \le C^{(2)}_{\nu} A_1,
\end{align}
where $C^{(2)}_{\nu}>0$ depends only on $\nu$. 

\item We have
\begin{align} \label{G5.9}
\sup_{\frac 1 {\tau} \le n \le \frac {10}{\tau}}
( \| u^n \|_{H^{40}} + \|v^n \|_{H^{40}}  ) \le  C^{(3)}_{\nu, A_1},
\end{align}
where $C^{(3)}_{\nu,A_1}>0$ depends only on ($\nu$, $A_1$).

\item If $\|a \|_{H^1} \le \tilde A_1$ for some $\tilde A_1<\infty$, then we also have
\begin{align}
\sup_{1\le n\le \frac {10}{\tau} } ( \| u^n \|_{H^1} + \| v^n \|_{H^1} ) \le 
\tilde C^{(3)}_{\nu, \tilde A_1},
\end{align}
where $\tilde C^{(3)}_{\nu, \tilde A_1}>0$ depends only on ($\nu$, $\tilde A_1$).

\end{enumerate}
\end{lem}
\begin{proof}
That the iterates $u^n$, $v^n$ are well-defined along with the estimate \eqref{G5.8} 
follows from Lemma \ref{lemG1}. 
A key observation here is that $\|u^n\|_2$, $\|v^n\|_2$ remains $O(1)$ for $n\tau \lesssim 1$.
 It suffices for us to show \eqref{G5.9} for $u^n$ since the
estimates for $v^n$ follow from it.  To this end we rewrite
\begin{align}
u^{n+1} & =S_L(\frac {\tau}2) S_N (\tau ) S_L(\frac {\tau}2) u^n \notag \\
& = S_L(\frac {\tau}2 ) ( S_{L}(\frac {\tau}2) u^n + \tau \partial_{xx} f^n ) ),\label{G5.9a}
\end{align}
where 
\begin{align}
f^n = \frac 1 {\tau} \int_0^{\tau} w_n (s)^3  ds,
\end{align}
and $w_n$ solves the PDE
\begin{align}
\begin{cases}
\partial_t w_n = \partial_{xx} (  w_n^3 ), \quad 0<t \le \tau; \\
w_n \Bigr|_{t=0} = S_L(\frac {\tau} 2) u^n.
\end{cases}
\end{align}
By Lemma \ref{lemG1}, we have 
\begin{align}
\sup_{n\tau \le 10} \|f^n \|_{\frac 43} \lesssim 1+\tau^{-\frac 3{16}}.
\end{align}
Iterating \eqref{G5.9a}, we obtain
\begin{align}
u^{n+1} = S_L( (n+1)\tau) u^0 + \tau \sum_{k=0}^n  \partial_{xx} S_L( (k+1)\frac 12\tau)    f^{{n-k} }. 
\end{align}
The desired estimates then follow from bootstrapping smoothing estimates.
\end{proof}

\begin{lem}[Almost steady states are benign] \label{lemG3}
Let $\nu>0$. Suppose $f \in H^2(\mathbb T)$ has zero mean and satisfies
\begin{align}
\| \nu \partial_{xx} f - f^3 +\overline{f^3} + f \|_2 \le 1,
\end{align}
where  $\overline{f^3}$ denotes the average of
$f^3$ on $\mathbb T$. Then 
\begin{align}
\| f \|_{H^{40}(\mathbb T)} \le C^{(4)}_{\nu},
\end{align}
where $C^{(4)}_{\nu}>0$ depends only on $\nu$.  Furthermore if $0<\tau \le \tau^{(0)}(\nu)$
where $\tau^{(0)} (\nu)>0$ is a sufficiently small constant depending only on $\nu$, then
\begin{align} \label{Cnu5}
E( S_L(\frac {\tau}2) S_N(\tau) S_L(\frac{\tau} 2) f ) \le C^{(5)}_{\nu},
\end{align}
where $C^{(5)}_{\nu}>0$ depends only on $\nu$.
\end{lem}
\begin{rem}
We pick the constant $1$ for convenience. If $\| \nu \partial_{xx} f - f^3 +\overline{f^3} + f \|_2 \le 
\epsilon_0$ for some $\epsilon_0>0$, then we obtain
$\|f\|_{H^{40}(\mathbb T)} \le C^{(4)}_{\nu, \epsilon_0}$, where $C^{(4)}_{\nu,\epsilon_0}>0$
depends on ($\nu$, $\epsilon_0$).
\end{rem}
\begin{proof}
By a simple energy estimate, we have
\begin{align}
\nu \| \partial_x f \|_2^2 + \|f^2 -\frac 12 \|_2^2 \le  \|f\|_2 + \frac {\pi}2.
\end{align}
One can then obtain the $H^1$ bound. Bootstrapping yields the desired estimate for
higher Sobolev norms. The estimate  \eqref{Cnu5} also follows easily.
\end{proof}

\begin{lem}[One-step strict energy dissipation for non-steady data] \label{lemG4}
Let $\nu>0$. Suppose $a\in H^{40}(\mathbb T)$ and has zero mean. Assume
\begin{align}
& \| \nu \partial_{xx} a -a^3 + \overline{a^3}+a \|_{2} \ge 1, \notag \\
& \| a\|_{H^{40}(\mathbb T)} \le B_1<\infty, 
\end{align}
where  $B_1$ is a  given constant, and $\overline{a^3}$ denotes
 the average of $a^3$ on $\mathbb T$. There exists $\tau_3=\tau_3(\nu, 
B_1)>0$ sufficiently small such that if $0<\tau \le \tau_3$, then
\begin{align}
E( S_L(\frac {\tau}2) S_N(\tau) S_L(\frac {\tau}2) a) < E(a).
\end{align}
\end{lem}
\begin{proof}
Denote by $u^{\mathrm{P}}$ as the exact PDE solution corresponding to initial data
$a$. We clearly have
\begin{align}
E(u^{\mathrm P} (\tau) ) + \int_0^{\tau}
\left\| \partial_x ( \nu \partial_{xx} u^{\mathrm{P}}  -
(u^{\mathrm{P} } )^3  +u^{\mathrm{P}}  ) \right\|_2^2 ds  = E(a).
\end{align}
By Poincar\'e we have
\begin{align}
\left\| \partial_x ( \nu \partial_{xx} u^{\mathrm{P}}  -
(u^{\mathrm{P} } )^3  +u^{\mathrm{P}} ) \right\|_2
\ge \left\|  \nu \partial_{xx} u^{\mathrm{P}}  -
(u^{\mathrm{P} } )^3   + \overline{  (u^{\mathrm{P} } )^3 } +u^{\mathrm{P}} \right\|_2.
\end{align}

Thanks to the high regularity assumption on $a$ and the usual local theory, we have
\begin{align}
\sup_{0\le s\le \tau} \left\| \nu \partial_{xx} u^{\mathrm{P}} (s)  -
(u^{\mathrm{P} } (s)  )^3  + \overline{ (u^{\mathrm{P} }(s) )^3}  +u^{\mathrm P}(s)
 -      \nu \partial_{xx} a +a^3 - \overline{a^3} -a \right\|_2 \lesssim \tau. 
\end{align}

It follows that 
\begin{align}
E(u^{\mathrm{P} } (\tau) ) + \tau \| \nu \partial_{xx} a- a^3 + \overline{a^3} +a\|_2^2
+O(\tau^2) \le E(a).
\end{align}
Thus for $\tau>0$ sufficiently small we have
\begin{align}
E(u^{\mathrm{P} } (\tau) ) +\frac 12 \tau  \le E(a).
\end{align}

We now only need to check that
\begin{align}
\| u^{\mathrm P}(\tau) - S_L(\frac {\tau} 2) S_N(\tau) S_L(\frac{\tau} 2) a \|_{H^1(\mathbb T)}
=O(\tau^2).
\end{align}

Consider an implicit-explicit discretization:
\begin{align}
\frac {w -a}{\tau} = - \nu \partial_x^4 w - \partial_x^2 w + \partial_x^2 (a^3).
\end{align}
It is not difficult to check that 
\begin{align}
\| u^{\mathrm P}(\tau) -w \|_{H^1}  =O(\tau^2)
\end{align}
and 
\begin{align}
w= (1+\nu \tau \partial_x^4 + \tau \partial_x^2)^{-1} a +
\tau (1+\nu \tau \partial_x^4 + \tau \partial_x^2)^{-1} \partial_{xx} (a^3).
\end{align}

On the other hand, we have
\begin{align}
\| S_L(\frac {\tau} 2) S_N(\tau) S_L(\frac {\tau} 2) a
- S_L(\frac {\tau} 2) ( S_L(\frac {\tau} 2 ) a + \tau \partial_{xx} (a^3 ) ) \|_{H^1}
=O(\tau^2).
\end{align}

The desired result then follows from the estimates
\begin{align}
& \| (1+\nu \tau \partial_x^4 + \tau \partial_x^2)^{-1} a  - S_L(\tau) a \|_{H^1} = O(\tau^2); 
\notag \\
& \| (1+\nu \tau \partial_x^4 + \tau \partial_x^2)^{-1} \partial_{xx} (a^3)
- S_L(\frac {\tau} 2)  \partial_{xx} (a^3) \|_{H^1} =O(\tau). \notag
\end{align}
\end{proof}

\begin{thm} \label{thm3.1}
Let $\nu>0$. Assume $u^0 \in H^1(\mathbb T)$ with zero mean. Suppose
$\|u^0\|_2 \le \gamma_0$ and $\| u^0 \|_{H^1} \le \gamma_1$.  Define
\begin{align}
u^{n+1} = S_L(\frac {\tau}2) S_N(\tau) S_L(\frac {\tau}2) u^n, \quad n\ge 0.
\end{align}
There exists $\tau_*=\tau_*(\nu,\gamma_0)>0$ sufficiently small such that if
$0<\tau \le \tau_*$, then
\begin{align} \label{G5.31}
\sup_{n\ge 1} \| u^n \|_{H^1(\mathbb T)} \le F^{(0)}_{\nu, \gamma_1},
\end{align}
where $F^{(0)}_{\nu, \gamma_1} >0$ depends only on ($\nu$, $\gamma_1$).
\end{thm}

\begin{proof}
By Lemma \ref{lemG2}, for some $\tilde \tau_1(\nu,\gamma_0)>0$
sufficiently small and  $0<\tau \le \tilde \tau_1(\gamma,\gamma_0)$,
we have
\begin{align}
\sup_{1\le n \le \frac {10}{\tau} } \| u^n \|_2
+ \sup_{\frac 1 {\tau} \le n \le \frac {10}{\tau}} (\| u^n \|_{H^{40} } +E(u^n) )
\le F^{(1)}_{\nu, \gamma_0},
\end{align}
where $F^{(1)}_{\nu,\gamma_0}$ depends only on ($\nu$, $\gamma_0$).

Define
\begin{align}
G_{\nu, \gamma_0} = F_{\nu, \gamma_0}^{(1)} +C_{\nu}^{(5)}+1,
\end{align}
where $C_{\nu}^{(5)}>0$ is the same as in \eqref{Cnu5}.

\underline{Claim}: For some $\tilde \tau_2(\nu,\gamma_0)>0$ sufficiently small
and $0<\tau \le \tilde \tau_2(\nu,\gamma_0)$, we have
\begin{align}
\sup_{n\ge \frac 1 {\tau}} E(u^n) \le  G:=G_{\nu, \gamma_0},
\end{align}

To prove the claim we argue by contradiction.  The smallness condition 
$0<\tau \le \tilde  \tau_2(\nu, \gamma_0)$ will be assumed in the argument below.
The needed smallness of $\tilde \tau_2(\nu, \gamma_0)$ can be easily worked out
from the conditions specified in the used  lemmas  such as Lemma \ref{lemG3} and so on.

Suppose $n_0\ge \frac 1 {\tau}$ is the 
first integer such that 
\begin{align} \label{G5.35}
E(u^{n_0}) \le G, \quad E(u^{n_0+1}) >G.
\end{align}
Clearly by our choice of $G$, we have $n_0\ge \frac {10}{\tau}$.  By Lemma \ref{lemG3}
we must have
\begin{align} \label{G5.36a}
\| \nu \partial_{xx} u^{n_0} - (u^{n_0})^3 + \overline{ (u^{n_0})^3}
+ u^{n_0} \|_2 >1.
\end{align}
Since $n_0\ge \frac {10}{\tau}$,  we have $E(u^{n_0- j_0}) \le G$ for some integer
$\frac 1 {\tau} \le j_0 \le \frac {1} {\tau}+2$.  By using smoothing estimates we obtain
\begin{align} \label{G5.36b}
\| u^{n_0} \|_{H^{40}(\mathbb T)} \le C_{\nu, G},
\end{align}
where $C_{\nu, G}>0$ depends on ($\nu$, $G$). Since $G$ depends on ($\nu$, $\gamma_0$),
we have $C_{\nu, G}$ depends only on ($\nu$, $\gamma_0$). 
By \eqref{G5.36a}, \eqref{G5.36b} and Lemma \ref{lemG4}, we obtain for sufficiently small
$\tau$ that
\begin{align}
E(u^{n_0+1} ) < E(u^{n_0})
\end{align}
which is clearly a contradiction to \eqref{G5.35}. Thus we have proved the claim.
Finally the estimate \eqref{G5.31} follows from a uniform $H^1$ estimate of $\|u^n \|_{H^1}$ 
for $1\le n \le \frac 1{\tau}$ using the condition $\|u^0\|_{H^1} \le \gamma_1$. 
\end{proof}
\section{Proof of Theorem \ref{thm0} and \ref{thm1}}
\begin{proof}[Proof of Theorem \ref{thm0}]
The $H^1$ estimate follows from Theorem \ref{thm3.1}. Higher order estimates follow
from the smoothing estimates.
\end{proof}
\begin{proof}[Proof of Theorem \ref{thm1}]
Thanks to the uniform $H^{40}$ estimates on the numerical
solution and the exact PDE solution, we only need to check consistency.  To simplify the notation,
we shall denote
\begin{align}
L= - \nu \partial_{x}^4  - \partial_{x}^2.
\end{align}

\underline{Consistency for the propagator $S_L(\frac{\tau}2) S_N(\tau) S_L(\frac {\tau}2)$}. 
We first note that if 
\begin{align}
\begin{cases}
\partial_t w = \partial_{xx}(w^3), \quad 0<t\le \tau; \\
w\Bigr|_{t=0} = b,
\end{cases}
\end{align}
where $w$ admits uniform control of its Sobolev norm, then 
\begin{align}
w(\tau) = b + \tau \partial_{xx}(b^3) + \frac 12 \tau^2
\partial_{xx} (3b^2 \partial_{xx}(b^3) ) + O(\tau^3).
\end{align}
Now if $u= S_L(\frac {\tau}2 ) S_N(\tau) S_L(\frac {\tau}2) a$, then 
(below $b= S_L(\frac {\tau}2) a$)
\begin{align}
u &=S_L(\frac{\tau}2) \Bigl( b + \tau \partial_{xx}(b^3) + \frac 12 \tau^2
\partial_{xx} (3b^2 \partial_{xx}(b^3) ) \Bigr) + O(\tau^3) \notag \\
& = S_L(\tau) a +  \tau S_L(\frac {\tau}2) 
\partial_{xx} ( (a+ \frac 12 \tau L a )^3 ) + \frac 12 \tau^2 
\partial_{xx} (3a^2 \partial_{xx}(a^3) ) + O(\tau^3) \notag \\
& = S_L(\tau ) a+  \tau \partial_{xx} (a^3)
+ \frac 12 \tau^2 \partial_{xx} \Bigl(  L(a^3) + 3a^2 (La+\partial_{xx} (a^3) ) \Bigr)+ O(\tau^3).
\end{align}

\underline{Reformulation of the exact PDE solution}. 
Let $u^{\mathrm{P}}$ be the exact PDE solution to \eqref{1} with initial data $\tilde a$.  We have
\begin{align}
u^{\mathrm{P}} (\tau) & = S_L(\tau) \tilde a + \int_0^{\tau} S_L(\tau -s) \partial_{xx} (u(s)^3) ds \notag \\
& =S_L(\tau)\tilde a + \int_0^{\tau} ( 1+ (\tau-s) L) \partial_{xx} (u(s)^3) ds + O(\tau^3) \notag \\
& = S_L(\tau) \tilde a+ \int_0^{\tau} \partial_{xx} 
\Bigl( \bigl(\,  \tilde a + s( L\tilde a + \partial_{xx}({\tilde a}^3) ) \,\bigr)^3        
\Bigr) ds + 
\int_0^{\tau} (\tau-s)L \partial_{xx}({\tilde a}^3) ds + O(\tau^3) \notag \\
& = S_L(\tau) \tilde a+ \tau \partial_{xx} ({\tilde a}^3)
+ \frac 12 \tau^2 \partial_{xx} \Bigl(  L({\tilde a}^3) + 3{\tilde a}^2 (L\tilde a+\partial_{xx} ({\tilde a}^3) ) \Bigr)+ O(\tau^3).
\end{align}
The desired estimate then clearly follows.
\end{proof}

\frenchspacing
\bibliographystyle{plain}

\end{document}